\documentclass[12pt]{amsart}
\usepackage{amssymb,latexsym, amsmath, amscd, array, graphicx}

\swapnumbers \numberwithin{equation}{section}

\theoremstyle{plain}

\newtheorem{thm}{Theorem}[section]

\newtheorem{lemma}[thm]{Lemma}
\newtheorem{conjec}[thm]{Conjecture}
\newtheorem{prop}[thm]{Proposition}
\newtheorem{cor}[thm]{Corollary}

\theoremstyle{definition}
\newtheorem{defin}[thm]{Definition}

\newtheorem{remark}[thm]{Remark}

\newtheorem{problem}[thm]{Problem}

 \newcommand{\Wi}{\widetilde}

\DeclareMathOperator{\trf}{{\rm trf}}

\def\pt{\protect\operatorname{pt}}

\def\Z{{\mathbb Z}}
\def\Q{{\mathbb Q}}

\def\1{\hbox{\rm\rlap {1}\hskip.03in{\rom I}}}
\def\Bbbone{{\rm1\mathchoice{\kern-0.25em}{\kern-0.25em}
{\kern-0.2em}{\kern-0.2em}I}}

\long\def\forget#1\forgotten{} %

\newcommand\ver[1]{\marginpar{\tiny Changed in Ver \VER}}

\date{\today}
\begin{document}

\title[macroscopic dimension]{On Gromov's positive scalar curvature conjecture for duality groups}

\author[A.~Dranishnikov]{Alexander  Dranishnikov} %

\address{Alexander N. Dranishnikov, Department of Mathematics, University
of Florida, 358 Little Hall, Gainesville, FL 32611-8105, USA}
\email{dranish@math.ufl.edu}
\thanks{Supported by NSF, grant DMS-1304627}


\begin{abstract}
We prove the inequality $$ \dim_{mc}\Wi M\le n-2$$ for the macroscopic dimension of the universal
covers $\Wi M$ of almost spin $n$-manifolds $M$ with positive scalar
curvature whose fundamental group $\pi_1(M)$ is a virtual duality group that satisfies the coarse
Baum-Connes conjecture.
\end{abstract}

\maketitle \tableofcontents

\section {Introduction}

In the book dedicated to Gelfand's 80th anniversary~\cite{Gr1} Gromov
introduced the notion of macroscopic dimension to formulate his observation that the universal cover
$\Wi M$ of a closed $n$-manifold with a positive
scalar curvature is dimensionally deficient on large scales:
\begin{conjec}[Gromov~\cite{Gr1}]
The macroscopic dimension of the
universal cover $\Wi M$ of a closed $n$-manifold with a positive
scalar curvature metric is at most $n-2$: $$\dim_{mc}\Wi{M^n}\le n-2.$$
\end{conjec}
This conjecture was proven
by D. Bolotov for 3-manifolds~\cite{Bo}. Then it was proved  by
Bolotov and the author~\cite{BD} for spin $n$-manifolds, $n>3$ whose
fundamental group satisfies the Analytic Novikov Conjecture and the
Rosenberg-Stolz condition~\cite{RS}: {\em The natural transformation of the 
connective real K-theory of the group to periodic,
$ko_*(\pi)\to KO_*(\pi)$ is a monomorphism}. Perhaps the only known class of 
groups for which Rosenberg-Stolz condition is satisfied 
consists of the products of free groups.

In this paper we prove the conjecture for spin manifolds whose
fundamental group is the virtual duality group that satisfies the Analytic Novikov Conjecture.
Moreover, we prove it for almost spin manifolds, i.e. manifolds with spin universal cover,
with the virtual duality fundamental group that satisfies the coarse Baum-Connes conjecture.
It allows substantially enlarge the class of groups for which the Gromov conjecture holds true
by adding to the (virtual) products of free groups covered in~\cite{BD}, virtually nilpotent groups, arithmetic groups, knot groups, braid groups, 
mapping class groups, $Out(F_n)$, and their products.

A weak version of the Gromov conjecture states that $$\dim_{mc}\Wi{M^n}\le n-1$$ 
for positive scalar curvature manifolds $M$.
It first appeared in~\cite{Gr2} in a different language. 
Even the weak Gromov's conjecture is extremely difficult, since it implies the famous 
Gromov-Lawson conjecture~\cite{GL}\cite{R2}:
{\em A closed aspherical manifold cannot carry a metric with positive scalar curvature}.
In this paper we show that
the condition $\dim_{mc}\Wi{M^n}\le n-1$ is the integral version of the notion of  macroscopically small manifolds
introduced by Gong and Yu~\cite{GY}.

In~\cite{Dr2},~\cite{Dr3} we proved the weak Gromov conjecture for all rationally
inessential orientable closed $n$-manifold $M$ whose fundamental
group is a duality group.  Note that in the case of a spin manifold with
a positive scalar curvature metric the rational inessentiality
follows from Rosenberg's theorem and the KO-homology Chern-Dold
character. In this paper we extend this result proving  the weak Gromov conjecture
for almost spin closed $n$-manifold $M$, $n\ge 5$ whose fundamental
group satisfies virtually the following coarse version of  the
Rosenberg-Stoltz condition: {\em The natural transformation of the connective K-theory 
to periodic $ku_*^{lf}(E\pi)\to K_*^{lf}(E\pi)$ is a
monomorphism} where $E\pi$ is the universal cover of the classifying space $B\pi=K(\pi,1)$. 
We note that the duality groups satisfies this condition.

In the last section of the paper we reformulate the property $$\dim_{mc}\Wi{M^n}\le n-1$$ 
in terms of the Stone-\v Cech compactification
$\beta\Wi{M^n}$ of $\Wi{M^n}$. Then
we present the reduction of the weak Gromov conjecture to
the rationality conjecture for certain group homology classes and the reduction of the later conjecture to
a conjecture about the cohomology classes of $\beta\Wi{M^n}$.

\section{Gromov's macroscopic dimension}
Here is the definition of the macroscopic dimension.
\begin{defin}\cite{Gr1}
A metric space $X$ has the {\em macroscopic dimension} less or equal
to $k$, $\dim_{mc}X\le k$, if there is a continuous uniformly
cobounded map $f:X\to N^k$ to a $k$-dimensional simplicial complex.
\end{defin}
A map $f:X\to Y$ of a metric space is {\em uniformly cobounded} if
there is a constant $C>0$ such that $diam(f^{-1}(y))<C$ for all
$y\in Y$.

In~\cite{Dr1} this definition was modified to $\dim_{MC}X$ by asking $f$ to be Lipschitz
for a uniform metric on $N^k$. This modification was useful for attacking another conjecture of Gromov.
Note that the inequalities
$$
\dim_{mc}X\le\dim_{MC}X\le asdim X$$ 
follows from the definitions. Here $asdim$ is Gromov's asymptotic dimension~\cite{Gr3},\cite{BeD}.
As it follows from~\cite{Dr4} there is an example that makes both inequality strict. In particular, every essential
$n$-manifold with the fundamental group $\Z^m$ for $m>n$ is such.

\subsection{Inessential manifolds} Gromov calls an $n$-manifold $M$ {\em inessential} if a 
map $f:M\to B\pi$ that classifies its universal covering
$p:\Wi M\to M$ can be deformed to a map into the $(n-1)$-dimensional skeleton, $g:M\to B\pi^{(n-1)}$. 
Otherwise a manifold is called {\em essential}.
Clearly, for every inessential $n$-manifold
$M$ we obtain $\dim_{mc}\Wi M<n$. Indeed, a lift $\Wi g:\Wi M\to E\pi^{(n-1)}$ of $g$ defines a 
required map to $(n-1)$-dimensional complex.

The following characterization of inessential manifolds can be found in~\cite{Ba} (see also~\cite{BD}, Proposition 3.2).
\begin{thm}\label{bab} For a closed oriented $n$-manifold $M$ the following are equivalent:

(1) $M$ is inessential;

(2) $f_*([M])=0$ in $H_n(B\pi;\mathbb Z)$ where $[M]\in H_n(M;\mathbb Z)$ is the fundamental class and $f:M\to B\pi$ 
classifies the universal cover of $M$.
\end{thm}
In the light of this theorem an oriented manifold is called {\em rationally inessential} if $f_*([M])=0$ in $H_n(B\pi;\Q)$~\cite{Gr2}.

We note that the proof of the inequality $\dim_{mc}\Wi M\le n-2$ for the cases covered in~\cite{BD} 
consists of constructing a deformation of the classifying map
$f:M\to B\pi$ to the $(n-2)$-dimensional skeleton $B\pi^{(n-2)}$. In this paper we construct a bounded 
deformation of the lift $\Wi f:\Wi M\to E\pi$ to
$E\pi^{(n-2)}$. The proof in~\cite{BD} was performed in the frame of the classical obstruction theory with twisted coefficients. 
The proof in this paper is based on the obstruction theory that suits the problem. The formulation of the corresponding 
obstruction theory is performed in
the rest of this section and the next section. It is culminated by the following theorem analogous to Theorem~\ref{bab}:
\begin{thm}\label{small}
For a closed oriented $n$-manifold $M$  the following are equivalent:

1. $\dim_{mc}\Wi M<n$;

2. $\Wi f_*([\Wi M])=0$ in $H^{lf}_n(E\pi;\mathbb Z)$ where $[\Wi M]\in H^{lf}_n(\Wi M;\mathbb Z)$ is the fundamental class of $\Wi M$
and $\Wi f:\Wi M\to E\pi$ is a lift of the map $f:M\to B\pi$ 
classifying the universal cover of $M$. 
\end{thm}
Thus the condition $\dim_{mc}\Wi M<n$ can be considered as  a 'coarse inessentiality' of $M$ with its characterization
given by Theorem~\ref{small}  analogous to the characterization of the inessentiality in Theorem~\ref{bab}.

\subsection{Macroscopic dimension and coarse inessentiality.}

We call a cellular map $f:X\to Y$ {\em bounded} if there is $k\in\mathbb N$ such that for every cell $e\subset X$ the image
$f(e)$ intersects at most $k$ cells.

For a finitely presented group $\pi$ we can choose an Eilenberg-McLane complex $K(\pi,1)=B\pi$ to be 
locally finite and hence metrizable.
We consider a proper geodesic metric on $B\pi$ and lift it to the universal cover $E\pi$. We recall  that
every closed ball with respect to a proper metric is compact.
In particular, every ball in $E\pi$ is contained in a finite subcomplex. 
The metric space $E\pi$ is {\em uniformly contractible}, i.e., there is a function $\rho:\mathbb R_+\to\mathbb R_+$ 
such that every set of diameter
$\le r$ can be contracted to a point in a ball of diameter $\rho(r)$.

For a CW complex $X$ we consider the product CW complex structure on $X\times I$. Clearly, a cellular 
homotopy $H:X\times I\to E\pi$ is bounded if and only if if there is a uniform upper bound on the diameter
of paths $H(x\times I)$, $x\in X$.

The following proposition shows that the inequality $\dim_{mc}\Wi M<n$ for the universal cover of a 
$n$-manifold can be considered as a coarse version
of inessentiality.
\begin{prop}\label{deform} 
For a finite complex $M$ the following conditions are equivalent:

(1) $\dim_{mc}\Wi M<n$.

(2) A lift $\Wi f:\Wi M\to E\pi$ to the universal coverings of a map $f:M\to B\pi$ that induces an 
isomorphism of the fundamental groups
can be deformed to the $(n-1)$-dimensional skeleton $E\pi^{(n-1)}$ of $E\pi$ by a bounded homotopy.
\end{prop}
The proof of Proposition~\ref{deform} can be extracted from~\cite{Dr1}, Proposition 3.1.

\subsection{Primary obstruction to a bounded extension}
Here we recall some basic facts of the elementary obstruction
theory. Let $f:X\to Y$ be a cellular map that induces an isomorphism
of the fundamental groups. We want to deform the map $f$ to a map to
the $(n-1)$-skeleton $Y^{(n-1)}$. For that we consider the extension
problem $$X\supset X^{(n-1)} \stackrel{f}{\to} Y^{(n-1)},$$ i.e.,
the problem to extend $f:X^{(n-1)}\to Y^{(n-1)}$ continuously to a
map $\bar f:X\to Y^{(n-1)}$. The primary obstruction for this
problem $o_f$ is the obstruction to extend $f$ to the $n$-skeleton.
It lies in the cohomology group $H^n(X;L)$ where
$L=\pi_{n-1}(Y^{(n-1)})$ is the $(n-1)$-dimensional homotopy group
considered as a $\pi$-module for $\pi=\pi_1(Y)=\pi_1(X)$. The class $o_f$ is generated by the cocycle
$C_f:C_n(X)\to L$ defined by the obvious rule: $$C_f(e)=[f\circ\phi_e:S^{n-1}\to Y^{(n-1)}]\in\pi_{n-1}(Y^{(n-1)})$$
where $\phi_e:\partial D^n\to X^{(n-1)}$ is the attaching map of an $n$-cell $e$.
Clearly, $f$ can be extended if and only if $C_f=0$.

The obstruction theory says that an extension $g:X\to Y^{(n-1)}$ that agrees
with $f$ on the $(n-2)$-skeleton $X^{(n-2)}$ exits if and only if
$o_f=0$. The primary obstruction is natural: If $g:Z\to X$ is a
cellular map, then $o_{gf}=g^*(o_f)$. In particular, in our case
$o_f=f^*(o_1)$ where $o_1\in H^n(Y;L)$ is the primary obstruction to
the retraction of $Y$ to the $(n-1)$-skeleton.

\

Now in the above setting we consider the extension problem for the universal coverings
$$\Wi X\supset \Wi X^{(n-1)} \stackrel{\Wi f}{\to} \Wi Y^{(n-1)}.$$
We are looking for an extension $g:\Wi X\to \Wi Y^{(n-1)}$ which is on a bounded distance from $\Wi f: \Wi X\to \Wi Y$.
We call such extension {\em bounded}. In view of Proposition~\ref{deform} existence of such extension in the case when $Y=B\pi$ means exactly the inequality 
$\dim_{mc}\Wi X<n$. Clearly, the vanishing of the classical obstruction $o_{\Wi f}$ would not guarantee a bounded extension. We need to restrict the freedom in the choice
of cochains.
Note that the obstruction cochain $C_{\Wi f}$ is $\pi$-equivariant and its class in the equivariant cohomology
$H^n_{\pi}(\Wi X;L)=H^n(X;L)$ is exactly the class $o_f$. Of course  vanishing of $o_f$ would imply a bounded extension 
but it is not  a necessary condition. It turns out that the coarsely equivariant cochains do the required job.
\begin{defin} The group of coarsely equivariant cochains
$C_{ce}^n(\Wi X,L)$  consists of
homomorphisms $\phi: C_n(\Wi X)\to L$ such that the set 
$$\{\gamma^{-1}\phi(\gamma e)\mid \gamma\in\pi\}\subset L$$ 
is contained in a finitely generated subgroup of $L$ for every $n$-cell $e$.
\end{defin}
\begin{prop}\label{obstructiontheory}
Let $\Wi f:\Wi X\to\Wi Y$ be a lift to the universal coverings of a  cellular map
$f:X\to Y$ of a finite  complex to a locally finite that induces an
isomorphism of the fundamental groups.  Then 

(i) $\Wi f:\Wi X^{(n-1)}\to Y^{(n-1)}$ has bounded extension to $X^{(n)}$ if and only if $C_{\Wi f}=0$.

(ii) $\Wi f:\Wi X^{(n-2)}\to Y^{(n-1)}$ has bounded extension to $X^{(n)}$ if and only if $C_{\Wi f}=\delta\Psi$
where $\Psi$ is a coarsely equivariant cochain.
\end{prop}
\begin{proof}
The proof goes along the lines of a similar statement from the
classical obstruction theory. 

(i) Let $\pi=\pi_1(X)$. If $C_{\Wi f}=0$ then $\Wi f$ has a $\pi$-equivariant extension which is, clearly, bounded.
In the other direction the statement is trivial.

(ii) Let $C_{\Wi f}=\delta\Psi$ where
$\Psi:C_{n-1}(\Wi X)\to L$, $L=\pi_{n-1}(\Wi Y^{(n-1)})$, is a coarsely
equivariant homomorphism. For each $(n-1)$-cell $e$ of $X$ we fix a
lift $\Wi e\subset\Wi X^{(n-1)}$. Since $\Psi$ is coarsely equivariant, the
set $$\{\gamma^{-1}\Psi(\gamma\Wi e)\mid
\gamma\in\pi\}$$ spans a
finitely generated subgroup $G\subset L$. Therefore, $G$ is contained in the image of $\pi_{n-1}(F_{\Wi e})$ for
some finite subcomplex $F_{\Wi e}\subset\Wi Y^{(n-1)}$.

Like in the classical obstruction theory we define a map
$g_{\gamma}:\gamma\Wi e\to \gamma F_{\Wi e}\subset Y^{(n-1)}$,
$\gamma\in\pi$ on cells $\gamma\Wi e$ such that $g_{\gamma}$ agrees with
$\Wi f$ outside a small $(n-1)$-ball $B_{\gamma}\subset\gamma\Wi e$ and
the difference of the restriction of $\Wi f$ and $g_{\gamma}$  to $B_{\gamma}$
defines a map $$d_{\Wi f,g_{\gamma}}:S^{n-1}=B_{\gamma}^+\cup B_{\gamma}^-\to
\gamma F_{\Wi e}\subset Y^{(n-1)}$$ that represents the class
$-\Psi(\gamma\Wi e)\in \pi_{n-1}(\Wi Y^{(n-1)})$. 
The union $\cup_{\gamma\in\pi}g_{\gamma}$ defines a bounded map
$g:\Wi X^{(n-1)}\to\Wi Y^{(n-1)}$.  

The elementary obstruction theory implies that for every $n$-cell
$\sigma\subset\Wi X^{(n)}$ there is an extension $\bar
g_{\sigma}:\overline{\sigma}\to\Wi Y^{(n-1)}$ of
$g|_{\partial\sigma}$. Since $\partial\sigma$ is contained in finitely many $(n-1)$-cells, there are $e_1,\dots, e_k\subset X^{(n-1)}$
and $\gamma_1,\dots,\gamma_k\in\pi$ such that $$g(\partial\sigma)\subset \bigcup_{i=1}^k\gamma_iF_{\Wi e_i}=F.$$
Since $F$ is a finite complex, the group $\pi_{n-1}(F)$ is finitely generated over $\Z\pi_1(F)$. Hence
there is a finite subcomplex $W\subset \Wi Y^{(n-1)}$ containing $F$ such that for the homotopy groups $\pi_{n-1}$ we have
$ker (j_W)_*=ker j_*$ where $j_W:F\to W$ and $j:F\to Y^{(n-1)}$ are the inclusions.
Therefore, we may assume that $\bar g_{\sigma}(\sigma)\subset W$. Since 
$g(\partial\gamma\sigma)\subset\gamma Z$, we may assume that
$\bar g_{\gamma\sigma}(\gamma\sigma)\subset \gamma W$ for all translates of $\sigma$, $\gamma\in\pi$.
Thus, the resulting extension $\bar
g:\Wi X^{(n)}\to\Wi Y^{(n-1)}$ is bounded.

In the other direction, if there is a $k$-bounded map $\bar
g:\Wi X\to Y^{(n-1)}$ that coincides with $\Wi f$ on the
$(n-2)$-dimensional skeleton, then the difference cochain $d_{\Wi
f,\bar g}$ is coarsely equivariant. Indeed, the $k$-bounded maps
preserving a base point $\phi:S^{n-1}\to Y^{(n-1)}$
define a finitely generated subgroup
of $\pi_{n-1}(Y)$.  Then the formula $\delta d_{\Wi f,\bar g}=
C_{\bar g}-C_{\Wi f}$ and the fact that $o_{\bar g}=0$ imply that
$o_{\Wi f}=0$.
\end{proof}

\section{Coarsely equivariant cohomology}

Let $X$ be a CW complex and let $E_n(X)$ denote the set of its
$n$-dimensional cells. We recall that (co)homology groups of a CW complex
$X$ with coefficients in a $\pi$-module $L$ where $\pi=\pi_1(X)$
are defined by by means of the chain complex $Hom_{\pi}(C_n(\Wi X),L)$ where $\Wi X$ is the
universal cover of $X$ with the cellular structure induced from $X$.
The chain complex defining the homology groups $H_*(X;L)$ is $\{C_n(\Wi
X)\otimes_{\pi}L\}$. The resulting groups $H_*(X;L)$ and $H^*(X;L)$ do not depends on the
CW structure on $X$.

These groups can be interpreted as the equivariant (co)homology:
$$H_*(X;L)=H_*^{lf,\pi}(\Wi X;L)\ \ \ and\ \ \ H^*(X;L)=H_{\pi}^*(\Wi X;L).$$
The last equality is obvious since the equivariant cohomology groups
$H_{\pi}^*(\Wi X;L)$ are defined by equivariant cochains
$C_{\pi}^n(\Wi X),L\}$, i.e., homomorphisms $\phi:C_n(\Wi X)\to L$ such that
the set $$S_{\phi,c}=\{\gamma^{-1}\phi(\gamma c)\mid \gamma\in\pi\}$$ consists of one element  for
every $c\in C_n(\Wi X)$. Here we consider the group of coarsely equivariant cochains $C_{ce}^n(\Wi X,L)$.
Thus we have a chain of inclusions of cochain complexes
$$
C^*_{\pi}(\Wi X,L)\subset  C^*_{ce}(\Wi X,L)\subset C^*(\Wi X,L)
$$
which induces a chain of homomorphism of corresponding cohomology groups
$$
H^n(X;L)\stackrel{ec^*_X}\longrightarrow H^n_{ce}(\Wi X;L)\to H^n(\Wi X;L)
$$
where the first homomorphism is called the {\em equivariant coarsening homomorphism}.
The cohomology
groups $H^*_{ce}(\Wi X;L)$ are called the {\em coarsely equivariant
cohomology} of $\Wi X$ with coefficients in a $\pi$-module $L$.

\begin{remark}
In~\cite{Dr1} we defined the almost equivariant cohomology groups by means of almost equivariant cochains
$C_{ae}^n(\Wi X,L)$, i.e., homomorphisms $\phi:C_n(\Wi X)\to L$ for which the set
$S_{\phi,c}$ is finite for every chain $c\in C_n(\Wi X)$.
The inclusions 
$
C^*_{\pi}(\Wi X,L)\subset C^*_{ae}(\Wi X,L)\subset C^*_{ce}(\Wi X,L)
$
define the homomorphisms of corresponding cohomology groups
$$
H^n(X;L)\stackrel{pert^*_{X}}\longrightarrow H^n_{ae}(\Wi X;L)\stackrel{\alpha}\longrightarrow H^n_{ce}(\Wi X;L)
$$
where $ec^*_X=\alpha\circ pert^*_X$.
\end{remark}

Note that the groups $H^*_{ce}(\Wi X;L)$ do not depend on
the choice of a CW complex structure on $X$.
Also note that in the case when $L=\Z$ is a trivial module the cohomology
theory $H^*_{ce}(\Wi X;\Z)$ equals the standard (cellular) cohomology of
the universal cover $H^*(\Wi X;\Z)$.

A  proper cellular map $f:X\to Y$ that induces an
isomorphism of the fundamental groups lifts to a proper cellular map
of the universal covering spaces $\bar f:\Wi X\to\Wi Y$. The lifting
$\bar f$ defines a chain homomorphism $\bar f_*:C_n(\Wi X)\to
C_n(\Wi Y)$ and the homomorphism of coarsely equivariant cochains $\bar f^*:Hom_{ce}(C_n(\Wi
Y),L)\to Hom_{ce}(C_n(\Wi X),L)$. The latter defines a homomorphisms
of the coarsely equivariant cohomology groups
$$\bar f^*_{ce}:H^*_{ce}(\Wi Y;L)\to H^*_{ce}(\Wi X;L).$$

Suppose that $\pi$ acts freely on CW complexes $\Wi X$ and $\Wi Y$
such that the actions preserve the CW complex structures. We call a
cellular map $g:\Wi X\to \Wi Y$ {\em coarsely equivariant} if the set
$$\bigcup_{\gamma\in\pi}\{\gamma^{-1}g(\gamma e)\}$$
is contained in a finite subcomplex $K\subset \Wi Y$ for every cell $e$ in $\Wi
X$. 

\begin{prop}\label{induced1}
Let $f:X\to Y$ be a proper coarsely equivariant cellular map. Then the
induced homomorphism on cochains takes the coarsely equivariant
cochains to coarsely equivariant.
\end{prop}
We omit the proof since it is similar to the proof of Proposition 2.3 in~\cite{Dr1}.
Proposition~\ref{induced1} and the standard facts about cellular
chain complexes imply the following.
\begin{prop}\label{induced2}
Let $X$ and $Y$ be  complexes with free cellular actions of a group
$\pi$.

(A) Then every coarsely equivariant cellular map $f:X\to Y$ induces an
homomorphism of the coarsely equivariant cohomology groups
$$
f^*:H^*_{ce}(Y;L)\to H^*_{ce}(X;L).$$

(B) If two coarsely equivariant maps $f_1, f_2: X\to Y$ are homotopic
by means of a cellular coarsely equivariant homotopy
$H:X\times[0,1]\to Y$, then they induce the same homomorphism of the
coarsely equivariant cohomology groups, $f^*_1=f^*_2$.
\end{prop}

\subsection{Coarsely equivariant homology}
We recall that the equivariant locally
finite homology groups are defined by the complex of infinite {\em
locally finite invariant chains}
$$ C_n^{\pi}(\Wi X;L)=\{\sum_{e\in E_n(\Wi X)} \lambda_ee\mid
\lambda_{ge}=g\lambda_e,\ \lambda_e\in L\}.$$ The local finiteness
condition on a chain requires that for every $x\in\Wi X$ there is a
neighborhood such that the number of $n$-cells $e$ intersecting $U$
for which $\lambda_e\ne 0$ is finite. This condition is satisfied
automatically when $X$ is a locally finite complex. It is known 
that $H_*(X;L)=H_*^{lf,\pi}(\Wi X;L)$.

Similarly one can define the coarsely equivariant homology groups on a
locally finite CW complex by considering infinite  coarsely equivariant
chains.  Let $X$ be a complex with the fundamental group $\pi$ and
the universal cover $\Wi X$. We call an infinite chain $\sum_{e\in
E_n(\Wi X)} \lambda_ee$ {\em coarsely equivariant} if the group generated by the set
$\{\gamma^{-1}\lambda_{\gamma e}\mid\gamma\in\pi\}\subset L$ is
finitely generated for every cell $e$. As we already have mentioned, the complex
of equivariant locally finite chains defines equivariant locally
finite homology $H^{lf,\pi}_*(\Wi X;L)$. The homology defined by the
coarsely equivariant locally finite chain are called  {\em the coarsely
equivariant locally finite homology}. We use notation
$H^{lf,ce}_*(\Wi X;L)$.  Like in the case of cohomology
this definition can be carried out for the singular homology and it
gives the same groups. In particular the groups $H^{lf,ce}_*(\Wi
X;L)$ do not depend on the choice of a CW complex structure on $X$.
As in the case of cohomology for any complex $K$ there is an equivariant coarsening homomorphism
$$ ec_*^K:H_*(K;L)=H^{lf,\pi}_*(\Wi K;L)\to H_*^{lf,ce}(\Wi K;L).$$
Also, there is an analog of Proposition~\ref{induced2} for the
coarsely equivariant locally finite homology:
\begin{prop}\label{h-induced}
Let $X$ and $Y$ be  complexes with free cellular actions of a group
$\pi$.

(A) Then every coarsely equivariant cellular map $f:X\to Y$ induces
a homomorphism of the coarsely equivariant homology groups
$$
f_*:H_*^{lf,ce}(X;L)\to H_*^{lf,ce}(Y;L).$$

(B) If two coarsely equivariant maps $f_1, f_2: X\to Y$ are
homotopic by means of a cellular coarsely equivariant homotopy, then
they induce the same homomorphism of the coarsely equivariant
cohomology groups, $(f_1)_*=(f_2)_*$.
\end{prop}

Let $M$ be an oriented $n$-dimensional PL manifold with a fixed
triangulation. Denote by $M^*$ the dual complex. There is a
bijection between $k$-simplices $e$ and the dual $(n-k)$-cells $e^*$
which defines the Poincar\' e duality isomorphism. This bijection
extends to a similar bijection on the universal cover $\Wi M$. Let
$\pi=\pi_1(M)$. For any $\pi$-module $L$ the Poincar\' e duality on $M$
with coefficients in $L$ is given by the cochain-chain level by
isomorphisms
$$
Hom_{\pi}(C_k(\Wi M^*),L) \stackrel{PD_k}\longrightarrow
C_{n-k}^{lf,\pi}(\Wi M;L)$$ where $PD_k$ takes a cochain
$\phi:C_k(\Wi M^*)\to L$ to the following chain $\sum_{e\in
E_{n-k}(\Wi M)}\phi(e^*)e$. The family $PD_*$ is a chain isomorphism
which is also known as the cap product
$$
PD_k(\phi)=\phi\cap[\Wi M]$$ with the fundamental class $[\Wi M]\in
C_n^{lf,\pi}(\Wi M)$, where $[\Wi M]=\sum_{e\in E_n(\Wi M)} e$. We
note that the homomorphisms $PD_k$ and $PD_k^{-1}$ extend to the
coarsely equivariant chains and cochains:
$$
Hom_{ce}(C_k(\Wi M^*),L) \stackrel{PD_k}\longrightarrow
C_{n-k}^{lf,ce}(\Wi M;L).$$ Thus, the homomorphisms $PD_*$ define
the Poincar\' e duality isomorphisms $PD_{ce}$ between the coarsely
equivariant cohomology and homology. We summarize this in the
following
\begin{prop}\label{PD}
For any closed oriented $n$-manifold $M$ and any $\pi_1(M)$-module
$L$ the Poincar\' e duality forms the following commutative diagram:
$$
\begin{CD}
H^k(M;L) @ >ec^*_M>> H^k_{ce}(\Wi M;L)\\
@V{-\cap [M]}VV @V{-\cap[\Wi M]}VV\\
H_{n-k}(M;L) @>ec_*^M>> H_{n-k}^{lf,ce}(\Wi M;L).\\
\end{CD}
$$
\end{prop}

We note that the operation of the cap product for equivariant
homology cohomology automatically extends on the chain-cochain level
to the cap product on the coarsely equivariant homology and
cohomology. Then the Poincar\' e Duality isomorphism $PD_{ce}$ for $\Wi
M$ can be described as the cap product with the homology fundamental class
$[\Wi M]\in H^{lf}_n(\Wi M)$.

\subsection{Obstruction to the inequality $\dim_{mc}\Wi M^n<n$}
In view of Proposition~\ref{deform} we can reformulate Proposition~\ref{obstructiontheory} as follows:

\begin{thm}\label{obstr-dim}
Let $X$ be a finite $n$-complex with $\pi_1(X)=\pi$ and let $f:X\to
B\pi$ be a  map that induces an isomorphism of the fundamental
groups. Then $\dim_{mc}\Wi X<n$ if and only if the  obstruction $o_{\Wi f}\in H^n_{ce}(\Wi X;\pi_{n-1}(Y^{(n-1)}))$
is trivial.
\end{thm}
\begin{proof}
If $\dim_{mc}\Wi X<n$, then by Proposition~\ref{deform} there is a
a bounded cellular homotopy of $\Wi f:\Wi X\to E\pi$ to a
map $g:\Wi X\to E\pi^{(n-1)}$. By Corollary~\ref{uniform=almost},
the map $g$ is coarsely equivariant. Then by
Proposition~\ref{induced2}, $o_{\Wi f}=\Wi f^*(o_1)=g^*i^*(o_1)=0$
where $i:E\pi^{(n-1)}\subset E\pi$ is the inclusion and $o_1$ is the primary obstruction to bounded retraction of $E\pi$ onto $E\pi^{(n-1)}$.

We assume that $B\pi$ is a locally finite simplicial complex. If
$o_{\Wi f}=0$, then by Proposition~\ref{obstructiontheory} there is
a bounded cellular map $g:\Wi X\to E\pi^{(n-1)}$ which agrees with
$\Wi f$ on the $(n-2)$-skeleton $X^{(n-2)}$. The uniform contractibility of $E\pi$ implies that $g$ is homotopic to
$\Wi f$ by means of a bounded cellular homotopy.  Then Proposition~\ref{deform} implies the inequality $\dim_{mc}\Wi X<n$.
\end{proof}

\begin{prop}\label{uniform=almost} Let $X$ and $Y$ be universal covers of
finite and locally finite complexes respectively with the same fundamental group $\pi$.  Then a bounded cellular
homotopy $\Phi:X\times[0,1]\to Y$ of a coarsely
equivariant map is coarsely equivariant.
\end{prop}
We refer to~\cite{Dr1} for the proof.

\

\subsection{Proof of Theorem~\ref{small}.}

1. $\dim_{mc}\Wi M<n\Rightarrow  \Wi f_*([\Wi M])=0$. We may assume that $f:M\to B\pi$ is cellular  and Lipschitz 
for some metric CW complex structure on $B\pi$. If $\dim_{mc}\Wi M<n$, then by
Proposition~\ref{deform} there is a coarsely Lipschitz cellular
homotopy of $\Wi f:\Wi X\to E\pi$ to a map $g:\Wi X\to E\pi^{(n-1)}$
with a compact projection to $B\pi$. By
Proposition~\ref{uniform=almost}, it is coarsely equivariant. Then
by Proposition~\ref{induced2} it follows that $\Wi
f_*([\Wi M])=0$.

2. $\Wi f_*([\Wi M])=0\Rightarrow \dim_{mc}\Wi M<n$.   Let $o_{\Wi f}$ be the primary obstruction for
a bounded deformation of $\Wi M$ to $E\pi^{(n-1)}$. Note that $o_{\Wi f}=\Wi f^*(o_1)$ where
$o_1$ is the primary obstruction for bounded retraction of
$E\pi^{(n)}$ to $E\pi^{(n-1)}$. Then $\Wi f_*([\Wi M]\cap o_{\Wi f})=\Wi f_*([\Wi M])\cap o_1=0$.
Since
the induced homomorphism
$$\Wi f_*:
H_0^{lf,ce}(\Wi M;L) \to H_0^{lf,ce}(E\pi;L)$$
is an isomorphisms for every $\pi$-module $L$, we obtain $[\Wi M]\cap o_{\Wi f}=0$.
By the Poincar\' e Duality (Proposition~\ref{PD}), $o_{\Wi f}=0$.
We apply Theorem~\ref{obstr-dim} to obtain the inequality $\dim_{mc}\Wi M< n$. 

$\square$

The following statement is in spirit of Brunnbauer- Hanke results~\cite{BH}.
\begin{cor}\label{smallhomology}
For each $n$ there is a subgroup $H^{sm}_n(B\pi)\subset H_n(B\pi)$ of small classes such that 
for an orientable $n$-manifold $M$ its universal covering satisfies $\dim_{mc}\Wi M<n$ if and only if $f_*([M])\in H^{sm}_n(B\pi)$ where $f:M\to B\pi$ is a classifying map for $\Wi M$.
\end{cor}
\begin{proof}
Define $H_n^{sm}(B\pi)=ker\{ec^{\pi}_*:H_n(B\pi)\to H^{lf,ce}_n(E\pi)\}$. Since $\Wi f_*\circ ec^M_*=ec^{\pi}_*\circ f_*$, we  have that
$$\dim_{mc}\Wi M<n\ \Leftrightarrow\  [\Wi M]\in ker\Wi f_*\ \Leftrightarrow\ f_*([M])\in H^{sm}_n(B\pi).$$
\end{proof}

\section{Coarse homology}

The coarse homology groups $HX_*(Y)$ of a metric space $Y$ were defined by  Roe~\cite{Ro1},\cite{Ro3} by means an anti-\v Cech approximation of $X$.
An anti-\v Cech approximation of a metric space $Y$ is a sequence of uniformly bounded locally finite open covers of $Y$ with finite multiplicity
$$\mathcal U_1\prec \mathcal U_2\prec \mathcal U_3\prec\dots$$
such that  for every $i$ the diameter of elements of $\mathcal U_i$ less than the Lebesgue number of $\mathcal U_{i+1}$,
$$mesh(\mathcal U_i)< Leb(\mathcal U_{i+1}).$$ Then the refinement $\mathcal U_i\prec\mathcal U_{i+1}$ defines the simplicial
map $p_i:N(\mathcal U_i)\to N(\mathcal U_{i+1})$ of the nerves. The coarse homology of $Y$ are defined as the direct limit
of homology groups of the nerves:
$$
HX_k(Y)=\lim_{\rightarrow} \{H_k^{lf}(N(\mathcal U_i)),(p_i)_*\}.
$$
Let $q^i:Y\to N(\mathcal U_i)$ be the projection to the nerve defined by a partition of unity on $Y$ subordinated to $\mathcal U_i$.
For a metric CW complex $Y$ it induces the homomorphism $q^1_*:H_*^{lf}(Y)\to H^{lf}(N(\mathcal U_1))$ which defines  a natural homomorphism $c_Y:H^{lf}_*(Y)\to HX_*(Y)$ called {\em coarsening} (see~\cite{HR}). 
The coarse homology groups are invariant under the coarse equivalence and, in particular, under guasi-isometries. Thus, for a closed manifold $M$ with $\pi_1(M)=\pi$
there is a natural isomorphism $ HX_*(\pi)\to HX_*(\Wi M)$. 

\subsection{Macrsocopically small manifolds.}
The following is an integral version of Gong-Yu's definition of  macroscopically large manifolds~\cite{GY},\cite{NY}.
\begin{defin} The universal covering $\Wi M$ of a closed manifold $M$ with the lifted metric is called {\em macroscopically large} if $c_{\Wi M}([\Wi M])\ne 0$ in $HX_*(\pi)$ for integral coefficients. Otherwise we call $\Wi M$ {\em macroscopically small}.
\end{defin}
In the original Gong-Yu definition~\cite{GY},\cite{NY} the coefficients were not mentioned  though they were assumed to be rational. Latter Brunnbauer and Hanke~\cite{BH} proved that Gong-Yu's concept
of the rational macroscopically large manifolds coincides with  Gromov's notion of a manifold with the infinite rational filling radius~\cite{Gr2}.
In the following theorem we prove a similar statement integrally.

\begin{thm}\label{m-small}
For a closed $n$-manifold $M$ its universal cover $\Wi M$ is macroscopically small
 if and only if $M$, $\dim_{mc}\Wi M<n$.
\end{thm}
\begin{proof}
Suppose that an $n$-manifold $M$, $\dim_{mc}\Wi M<n$. By Proposition~\ref{deform} there is a uniformly cobounded bounded cellular map
$g:\Wi M\to L\subset E\pi^{(n-1)}$ to an $(n-1)$-dimensional complex $L$. Moreover, we may assume that $g$ is a quasi-isometry to $L$.
Therefore, $c_{\Wi M}=c_L\circ g_*$. Since $L$ is $(n-1)$-dimensional, $g_*([\Wi M])=0$. Thus, $\Wi M$ is macroscopically small.

Suppose that $c_{\Wi M}([\Wi M])= 0$. Then $q^i_*([\Wi M])=0$ for some $i$ where $q^i:\Wi M\to N(\mathcal U_i)$ is the projection to the nerve in an anti-\v Cech approximation of $\Wi M$. The uniform contractibility of $E\pi$ and finite-dimensionality of $N(\mathcal U_i)$ imply that there is a map $s:N(\mathcal U_i)\to E\pi$ such that the map $s\circ q^i$ is proper homotopic to $\Wi f$. Therefore, $\Wi f_*([\Wi M])=0$
and by Theorem~\ref{small}, $\dim_{mc}\Wi M<n$.
\end{proof}

We note that in~\cite{Dr4} we called manifolds $M$ with $\dim_{mc}\Wi M<n$ {\em md-small} where $md$ stands for macroscopic dimension.
In this paper in view of Theorem~\ref{m-small} we stick to Gong-Yu's terminology calling $\Wi M$  {\em macroscopically small}.

\begin{problem}\label{mac-large}
{\em Suppose that the universal covering $\Wi M$ of a closed manifold is macroscopically large. Does it follow that $\Wi M$ is macroscopically large rationally?}
\end{problem}

\subsection{Coarse Baum-Connes conjecture} We recall that for a generalized homology theory $h_*$ the locally finite homology groups $h_n^{lf}(X)$
of a locally compact space
are defined as the Steenrod $h_*$-homology of the one-point compactification $\alpha X$.

Roe introduced certain $C^*$-algebra $C^*_{Roe}(X)$ of an open Riemannian manifold $X$ and defined a coarse index map~\cite{Ro1},\cite{Ro2}\cite{R2}
$$\mathcal A_* : K_*^{lf}(X)\to K_*(C^*_{Roe}(X)).$$
The construction of $C^*_{Roe}(X)$ and $\mathcal A$ can be extended to any metric space $X$ \cite{HR}.
In particular, $C^*_{Roe}(X)$ is a coarse invariant. Note that the nerve of a uniformly bounded cover $\mathcal U$ can be given a metric such that
the projection $q:X\to N(\mathcal U)$ is a quasi-isometry.
Then using an anti-\v Cech approximation $\{\mathcal U_i\}$ of $X$ and taking the direct limit
of the coarse index maps $K_*^{lf}(N(\mathcal U_i))\to K_*(C^*_{Roe}(N(\mathcal U_i)))\cong K_*(C^*_{Roe}(X))$
one can defined a homomorphism
$$
\mathcal A_{\infty}:KX_*(X)\to K_*(C^*_{Roe}(X))
$$
called {\em the coarse assembly map}.

\begin{conjec}[Coarse Baum-Connes conjecture~\cite{Ro1},\cite{HR}] For a  proper metric space $X$ 
the coarse assembly map is an isomorphism.
\end{conjec}

We say that a finitely generated group $\pi$ satisfies the coarse Baum-Connes conjecture if $\pi$  satisfies the coarse Baum-Connes conjecture as a metric space with respect to the word metric for a finite set of generators. For the definition of the coarse $K$-homology in the case of discrete space
one can take a sequence of Rips complexes
$$
\pi\subset R_1(\pi)\subset R_2(\pi)\subset\dots\subset R_m(\pi)\subset\dots
$$
instead of an anti-\v Cech system. Then $$KX_*(\pi)=\lim_{\rightarrow}K_*^{lf}(R_m(\pi)).$$
We recall that the Rips complex $R_m(\pi)$ has $\pi$ as the set of vertices and a finite subset $\{\gamma_0,\dots \gamma_k\}\subset\pi$ spans a $k$-simplex in $R_m(\pi)$ if and only if
$d(\gamma_i,\gamma_j)\le m$ for all $0\le i,j\le k$.

\begin{prop}\label{c}
For every closed manifold $M$ with the fundamental group $\pi$ there is a 
homomorphism $b$ that makes a commutative diagram
$$
\begin{CD}
K_*^{lf}(\Wi M) @ >{c_{\Wi M}}>> KX_*(\Wi M)\\
@V\Wi f_*VV @V{\cong}VV\\
K_*^{lf}(E\pi) @<b<< KX_*(\pi).\\
\end{CD}
$$
The homomorphism $b$ is an isomorphism for geometrically finite groups $\pi$.
\end{prop}
\begin{proof}
We construct a sequence of proper maps $b_i:R_i(\pi)\to E\pi$ such that $b_{i+1}|_{R_i(\pi)}=b_i$ for all $i\in\mathbb N$ where 
$b_0:\pi\to E\pi$ is an orbit of a base point, $b_0(\gamma)=\gamma(x_0)$.
\end{proof}

\

\section{Gromov's scalar curvature conjecture}

We recall that the Gromov conjecture states that $\dim_{mc}\Wi M\le n-2$ for the universal cover of a closed $n$-manifold
with positive scalar curvature. Clearly, the conjecture has two stages where the first stage is to prove the inequality $\dim_{mc}\Wi M\le n-1$.
We call it the weak Gromov conjecture. 
In view of Theorem~\ref{m-small} the weak Gromov conjecture can be reformulated as following: {\em The universal covering of a manifold with
positive scalar curvature is macroscopically small.}

We call manifolds with spin universal
cover {\em almost spin}. 
The spin structure on $\Wi M$ defines  a K-theory orientation. Thus, there exists  the K-theory fundamental class
$[\Wi M]_K\in K_n^{lf}(\Wi M)$. We denote by $ku_*$ the connective K-theory. Then there is the fundamental class
$[\Wi M]_{ku}$ which is taken to $[\Wi M]_K$ under the transformation $ku_*\to KU_*=K_*$.

\subsection{Deformation to the $(n-1)$-skeleton.} Our deformation of $\Wi f:\Wi M\to E\pi$
to $E\pi^{(n-1)}$ is parallel to the  deformation of $f:M\to B\pi$ to $B\pi^{(n-1)}$ performed in~\cite{BD}.
In~\cite{BD} the main ingredients where the Analytic Novikov conjecture and the vanishing theorem or Rosenberg~\cite{R1}.
Here we use the coarse Baum-Connes conjecture and the vanishing theorem of Roe~\cite{Ro1},\cite{HR}.

\begin{thm}[Roe]\label{roe}
Suppose that a closed  almost spin manifold $M$ has positive scalar curvature.
Then the coarse index map
$$\mathcal A_* : K_*^{lf}(\Wi M)\to K_*(C^*_{Roe}(\pi))$$ 
takes the K-theory fundamental class $[\Wi M]_{K}$ to zero.
\end{thm}

Another parallel with~\cite{BD} is  the usage of the connective K-theory.

\begin{thm}\label{main-1}
Suppose that the fundamental group $\pi$ of a closed almost spin $n$-manifold $M$,
$n\ge 5$,  with a positive scalar curvature metric satisfies the coarse Baum-Connes conjecture
and  the natural transformation $ku_n^{lf}(E\pi)\to K_n^{lf}(E\pi)$ is injective. Then
the weak Gromov conjecture holds true: $\dim_{mc}\Wi M<n$.
\end{thm}
\begin{proof}
It suffices to show that $\Wi f_*([\Wi M])$ is zero in  $H_*^{lf}(E\pi)$.

Note that the coarse index map  for $\Wi M$ is the composition $$\mathcal A=\mathcal A_{\infty}\circ i_*\circ c_{\Wi M}$$ in the following 
commutative diagram
$$
\begin{CD}
K_*^{lf}(\Wi M) @>c_{\Wi M}>>KX_*(\Wi M) @.\\
@V{\Wi f_*}VV @V{i_*}V{\cong}V @.\\
 K^{lf}_*(E\pi) @<b<< KX_*(\pi) @>{\mathcal A_{\infty}}>> K_*(C^*_{Roe}(\pi)).
\end{CD}
$$
Since the coarse Baum-Connes conjecture holds for $\pi$, the coarse assembly map  $\mathcal A_{\infty}$ is an isomorphism. Then in view of Proposition~\ref{c} and Theorem~\ref{roe}
$\Wi f_*([\Wi M]_K)=0$ in $K_*^{lf}(E\pi)$. By the condition of the theorem we obtain $\Wi f_*([\Wi M]_{ku})=0$ in $ku_*^{lf}(E\pi)$.

We note that the natural transformation of the connected spectrum to the Eilenberg-McLane spectrum $ku\to H(\mathbb Z)$ takes the fundamental class $[\Wi M]_{ku}$ to the fundamental class $[\Wi M]$. Hence $\Wi f_*([\Wi M])=0$. By Theorem~\ref{small},
$\dim_{mc}\Wi M<n$.
\end{proof}

\subsection{Deformation to the $(n-2)$-skeleton.}
\begin{lemma}\label{n-1}
Let $\dim_{mc}\Wi M<n$ for an oriented closed $n$-manifold $M$ with a classifying map $f:M\to B\pi$. Then its lifting $\Wi f:\Wi M\to E\pi$ admits a bounded deformation to a map $g:\Wi M\to E\pi^{(n-1)}$ such that $g(\Wi M^{(n-1)})\subset E\pi^{(n-2)}$.
\end{lemma}
\begin{proof} First we show that the map $\Wi g=\Wi f|_{\Wi M^{(n-1)}}:\Wi M^{(n-1)}\to E\pi^{(n-1)}$ admits a bounded deformation to a map
$q:\Wi M^{(n-1)}\to E\pi^{(n-2)}$ that agrees with $\Wi f$ on $\Wi M^{(n-3)}$.
In view of Proposition~\ref{obstructiontheory} and Theorem~\ref{obstr-dim} it suffices to show that $o_{\Wi g}=0$ in $H^{n-1}_{ce}(\Wi M;\pi_{n-2}(E\pi^{(n-2)}))$.
Since $o_{\Wi g}=\Wi f^*(o_1)$ and $\Wi f_*([\Wi M])=0$, we obtain $$\Wi f_*([\Wi M]\cap o_{\Wi g})=\Wi f_*([\Wi M])\cap o_1=0.$$
Since $\Wi f_*$ is an isomorphism in dimension 1, $[\Wi M]\cap o_{\Wi g}=0$. By the Poincar\' e duality,
$o_{\Wi g}=0$.

Then we extend the map $q$ to a map $g:\Wi M\to E\pi^{(n-1)}$ which is in a finite distance to $\Wi f$. We may assume that $M$ has one $n$-cell
$e$. Let $\bar e$ be any lift of $e$ to $\Wi M$. Since $q$ is coarsely equivariant, there is a finite complex $K\subset E\pi$ such that $$\bigcup_{\gamma\in\pi}\gamma^{-1}q(\partial\gamma\bar e)\subset K.$$
Since $E\pi$ is contractible, there is a finite complex $F$ containing $K$ such that the inclusion $K\to F$ is nulhomotopic. 
Then the inclusion $K^{(n-2)}\subset F^{(n-1)}$ is nulhomotopic.
Let $\psi_{\gamma}:\gamma\bar e
\to F^{(n-1)}$ be an extension of $\gamma^{-1}q$ restricted to the boundary $\partial\gamma\bar e$ of the $n$-cell $\gamma\bar e\subset\Wi M$, $\gamma\in\pi$. Then the union of maps $\cup_{\gamma\in\pi}\gamma\psi_{\gamma}$ and $q$ defines the required map $g$.
\end{proof}

We note that the Atiyah-Hirzebruch spectral sequence converges for the Steenrod homology for a finite dimensional compact metric spaces~\cite{EH}.
Thus, for finite dimensional $X$ for any generalized homology theory $h_*$ there is the Atiyah-Hirzebruch spectral sequence for the locally finite homology which has the $E^2$-term $E^2_{p,q}=H^{lf}_p(X;h_q(\pt))$ and converges to $h_*^{lf}(X)$.

\begin{lemma}\label{n-2}
Suppose that $g:(\Wi M,\Wi M^{(n-1)})\to(E\pi^{(n-1)},E\pi^{(n-2)})$  is a coarsely equivariant bounded
cellular map of the universal cover of an almost 
spin manifold with the fundamental group $\pi$. Assume that  
$H^{lf}_i(E\pi)=0$ for $0<i<n-1$. Then $\dim_{mc}\Wi M\le n-2$.
\end{lemma}
\begin{proof}
First we show that for every proper map $g:\Wi M\to E\pi^{(n-1)}$ with $g(\Wi M^{(n-1)})\subset E\pi^{(n-2)}$ takes the $ko$-fundamental class to zero.
Since $H^{lf}_i(E\pi^{(n-1)})=0$ for $0<i<n-1$,
in the Atiyah-Hirzebruch spectral sequence for the $ko^{lf}$ we have $E^2_{i,n-i}=0$ for all $0<i<n-1$. Therefore, the image
$g_*([\Wi M]_{ko})$ lives in the image of the $E^2$-terms $im\{g_*:E^2_{n-1,0}(\Wi M)\to E^2_{n-1,0}(E\pi^{(n-1)})\}$. We show that the induced homomorphism
$$g_*:H^{lf}_{n-1}(\Wi M;ko_0(pt))\to H^{lf}_{n-1}(E\pi^{(n-1)};ko_0(pt))$$ is zero, then it would follow that $g_*([\Wi M]_{ko})=0$.

Note that for the one-point compactifications the quotient $\alpha\Wi M/\alpha\Wi M^{(n-1)}$ is homeomorphic to the one-point compactification of
an infinite wedge of $n$-spheres.
Then $H_{n-1}(\alpha\Wi M/\alpha\Wi M^{(n-1)})=0$.
From the Steenrod homology exact sequence of the pair $(\alpha\Wi M,\alpha\Wi M^{(n-1)})$  it follows that the inclusion homomorphism
$i_*:H_{n-1}^{lf}(\Wi M^{(n-1)})\to H_{n-1}^{lf}(\Wi M)$ is an epimorphism. Since $H_{n-1}^{lf}(E\pi^{(n-2)})=0$, the commutative diagram
$$
\begin{CD}
H_{n-1}^{lf}(\Wi M) @>g_*>> H_{n-1}^{lf}(E\pi^{(n-1)})\\
@A{i_*}AA @AAA\\
H_{n-1}^{lf}(\Wi M^{(n-1)}) @>>> H_{n-1}^{lf}(E\pi^{(n-2)})\\
\end{CD}
$$
implies that $g_*$ is zero homomorphism.

We construct a CW complex $L$ by replacing every $(n-1)$-dimensional cell $\sigma\subset E\pi^{(n-1)}$ by an $n$-cell $D_{\sigma}$ attached by the composition
$\phi_{\sigma}\circ \beta:S^{n-1}\to E\pi^{(n-2)}$ where $\beta:S^{n-1}\to S^{n-2}$ generates $\pi_{n-1}(S^{n-2})$. 
Let $\psi:L\to E\pi^{(n-1)}$ denote the map which is obtained by extending of the identity map on $E\pi^{(n-2)}$ by  the cones of the attaching maps
$\phi_{\sigma}\circ \beta$.

Since the inclusion $E\pi^{(n-2)}\subset E\pi^{(n-1)}$ is nulhomotopic, the image of $\pi_n(E\pi^{(n-1)},E\pi^{(n-2)})$ under the boundary homomorphism generates $\pi_{n-1}(E\pi^{(n-2)})$. Thus the group $\pi_{n-1}(E\pi^{(n-2)})$ is generated by the boundaries of $(n-1)$-simplices, $\pi_{n-1}(\partial\sigma)$. Therefore any map $\phi:S^{n-1}\to E\pi^{(n-2)}$ is nulhomotopic in $L$. 
Hence
the map
$g:\Wi M^{(n-1)}\to E\pi^{(n-2)}$ extends to a  coarsely equivariant bounded
cellular map $f:\Wi M\to L$. 

For evry $n$-cell $D_{\sigma}$ we may assume that there is  a regular value of $f$, i.e. there is a closed $n$-disk $ D'_{\sigma}\subset D_{\sigma}$, such that
$f^{-1}(D)=\coprod B_i$ and the restriction of $f$ to each $B_i$ is a diffeomorphism between $B_i$ and $D'_{\sigma}$ for all $i$. We join all $B_i$ by tubes in $\Wi M$
in one $n$-ball $B_{\sigma}$. We can do it in such a way that all $B_{\sigma}$ are disjoint and uniformly bounded.

We use the notation $\stackrel{\circ} B$ for the interior of $B$.
Let $r:L\to L$ denote a map that
deforms $D_{\sigma}\setminus
\stackrel{\circ}{D_{\sigma}'}$ to $E\pi^{(n-2)}$ and defines a homeomorphis of $\stackrel{\circ}{D'_{\sigma}}$ to $\stackrel{\circ}D_{\sigma}$ for all $\sigma$. Thus the restriction of $r$ to
$L\setminus(\cup_{\sigma}\stackrel{\circ}{D'_{\sigma}})$ is a retraction onto $E\pi^{(n-2)}$. 

If the degree $d_{\sigma}=\sum_ideg(f|_{B_i})$ is even,
then in view of the fact that $\pi_{n-1}(S^{n-2})={\mathbb Z}_2$ for $n\ge 5$, the map $r\circ f|_{\partial B_{\sigma}}$ is nulhomotopic.
Thus, if $d_{\sigma}$ is even for all $\sigma$, there is a bounded deformation of $g$ into $E\pi^{(n-2)}$.

Consider the commutative diagram generated by exact sequences of pairs and the maps $r\circ f$ and $\psi$:
$$
\begin{CD}
ko_n^{lf}(\Wi M) @>r_*f_*>> ko_n^{lf} (L) @>\psi_*>> ko_n^{lf}(E\pi^{(n-1)})\\
@Vj^1_*VV @Vj^2_*VV @Vj^3_*VV\\
ko_n^{lf}(\Wi M,\Wi M\setminus\stackrel{\circ}B_{\sigma}) @ >r_*f_*>> ko_n^{lf}(L,L\setminus\stackrel{\circ} D_{\sigma}) @>\psi_*>> ko_n^{lf}(E\pi^{(n-1)},E\pi^{(n-1)}\setminus\stackrel{\circ}\sigma)\\
@V\cong Vq_*^1V @V\cong VV @V\cong Vq^3_*V\\
ko_n(B_{\sigma}/\partial B_{\sigma}) @>d_n>> ko_n(D_{\sigma}/\partial D_{\sigma}) @>\Sigma\beta_*>> ko_n(\sigma/\partial\sigma).\\
\end{CD}
$$
As it already has been proved, the conditions of the lemma imply that $\psi_*r_*f_*([\Wi M]_{ko})=0$. Therefore, $\Sigma\beta_*\circ d_{\sigma}\circ q^1_*\circ j^1_*([\Wi M]_{ko})=0$.
Note that the suspension $\Sigma\beta: S^n\to S^{n-1}$ induces the epimorphism $\Sigma\beta_*:\mathbb Z\to{\mathbb Z}_2$ and
$q^1_*\circ j^1_*([\Wi M]_{ko})$ is a generator of $ko_n(B_{\sigma}/\partial B_{\sigma})=\mathbb Z$. Therefore $d_{\sigma}$ cannot be odd.

Thus, $d_{\sigma}$ is even  for all $\sigma$, the maps $r\circ f|_{\partial B_{\sigma}}$ are nulhomotopic, and
there is a bounded deformation of $g$ in $E\pi$ to a map $\bar f:\Wi M\to E\pi^{(n-2)}$. By Proposition~\ref{deform} $\dim_{mc}\Wi M\le n-2$.
\end{proof}

\

\subsection{Duality groups.}
We recall that a group $\pi$ is called a {\em duality group}~\cite{Br} 
if there is a $\pi$-module $D$ such that
$$
H^i(\pi,M)\cong H_{m-i}(\pi,M\otimes D)
$$
for all $\pi$-modules $M$ and all $i$ where $m=cd(\pi)$ is the cohomological dimension of $\pi$.
The  groups that admit finite $B\pi$ are called {\em geometrically finite}
or of the type $FL$.
\begin{prop}\label{duality}
Let $\pi$ be a FL duality group. Then $H^{lf}_i(E\pi;\mathbb Z)=0$ for all $i\ne cd(\pi)$.
\end{prop}
\begin{proof} 
From Theorem 10.1 of~\cite{Br} it follows that $H^i(\pi,\mathbb Z\pi)=0$ for $i\ne m= cd(\pi)$ and
$H^m(\pi,\mathbb Z\pi)$ is a free abelian group. In view of the equality $H^i(\pi,\mathbb Z\pi)=H_c^i(E\pi;\mathbb Z)$
for geometrically finite groups
(see~\cite{Br} Theorem 7.5) and the short exact sequence for the Steenrod homology of a compact metric space
$$
0\to Ext(H^{i+1}(X),\mathbb Z)\to H^s_i(X;\mathbb Z)\to Hom(H^i(X),\mathbb Z)\to 0
$$
applied to the one point compactification $\alpha(E\pi)$ of $E\pi$ we obtain that $H^s_i(\alpha(E\pi);\mathbb Z)=0$ for
$i< m$. The equality $H^{lf}_i(E\pi;\mathbb Z)=H^s_i(\alpha(E\pi);\mathbb Z)$ completes the proof.
\end{proof}
We note that every duality group $\pi$ has type $FP$~\cite{Br}, i.e., $B\pi$ is dominated by a finite complex. 
It is still an open problem whether $FP=FL$~\cite{Br}. A group $\pi$ is called virtually $FL$ if it contains a finite index subgroup $\pi'$ which is $FL$.
We note that all classes of virtual duality groups listed in the introduction are virtually $FL$.

\begin{thm}\label{main-2}
Suppose that the fundamental group $\pi$ of a closed almost spin $n$-manifold $M$ with positive scalar curvature,
$n\ge 5$, is a virtual duality FL group that satisfies the coarse Baum-Connes conjecture.
Then the Gromov Conjecture holds for $M$, $\dim_{mc}\Wi M\le n-2$.
\end{thm}
\begin{proof}
Let $\pi'$ be a finite index subgroup of $\pi$ which is a FL duality group. Since $\pi'$ is quasi-isometric to $\pi$, the coarse Baum-Connes conjecture holds
for $\pi'$. Let $M'\to M$ be a covering that corresponds to $\pi'$. Note that the metric on $M'$ lifted from $M$ has positive scalar curvature and
$\Wi M'=\Wi M$.

If $n<\dim\pi'$, then Proposition~\ref{duality} implies that $H_i^{lf}(E\pi')=0$ for $0<i\le n$. Thus, the condition
of Theorem~\ref{main-1} is satisfied. Then $\dim_{mc}\Wi M'<n$. We apply Lemma~\ref{n-1} and Lemma~\ref{n-2} to obtain the inequality $\dim_{mc}\Wi M' \le n-2$. 

If $n>\dim\pi +1$, the inequality $\dim_{mc}\Wi M'\le n-2$ holds automatically.

If $n=\dim\pi +1$ we apply Lemma~\ref{n-2}.

Consider the case $n=\dim\pi$. Let $\Wi f:\Wi M'\to E\pi$ be a lift of the classifying map. As in the proof of Theorem~\ref{main-1} we obtain that
$\Wi f_*([\Wi M']_K)=0$. From the Atiyah-Hirzebruch spectral sequence it follows that $\Wi f_*([\Wi M'])=0$ for the integral locally finite homology.
Then by Theorem~\ref{small}, $$\dim_{mc}\Wi M'<n.$$ Then we apply Lemma~\ref{n-1} and Lemma~\ref{n-2} to complete the proof.
\end{proof}

It is known that the coarse Baum-Connes conjecture implies the Analytic Novikov conjecture. For spin manifolds one can replace the coarse Baum-Connes
conjecture by the Analytic Novikov conjecture. We say that a group $\pi$ virtually satisfies the Analytic Novikov conjecture if it contains a finite index subgroup $\pi'$ that satisfies the Analytic Novikov conjecture.

\begin{thm}\label{main-3}
Suppose that the fundamental group $\pi$ of a closed  spin $n$-manifold $M$,
$n\ge 5$,  with positive scalar curvature is a virtual $FL$ duality group that virtually satisfies the Analytic Novikov conjecture.
Then $\dim_{mc}\Wi M\le n-2$.
\end{thm}
\begin{proof} As in the proof of Theorem~\ref{main-2} we may assume that $\pi$ is a $FL$ duality group that satisfies the Analytic Novikov conjecture.
We note that the K-theory fundamental class $[M]_K$ goes under the tranfer map $\trf_M:K_*(M)\to K_*^{lf}(\Wi M)$ to the fundsmental class $[\Wi M]_K$.
By the Analytic Novikov conjecture and Rosenberg's vanishing theorem~\cite{R1} it follows that $f_*([M]_K)=0$ where $f:M\to BK$ is a classifying map. Therefore, $\Wi f_*([\Wi M]_K)=\Wi f_*\trf_M([M]_K)=\trf_{B\pi} f_*([M]_K)=0$ for a lift $\Wi f$ of $f$.
The rest of the  proof is the same as in Theorem~\ref{main-2}.
\end{proof}
\begin{remark}
It is natural to expect an extension of Theorem~\ref{main-1} to the following: {\em If $\pi$ satisfies the coarse Baum-Connes conjecture
and $ko^{lf}_*(E\pi)\to KO^{lf}_*(E\pi)$ is injective then  the Gromov conjecture holds true for almost spin manifolds with the fundamental group $\pi$.}
The weak Gromov conjecture follows in this case similarly to the proof of Theorem~\ref{main-1}. The necessary tools to deal with the real K-theory are given in~\cite{Ro4}. Still, the proof of  vanishing of the second obstruction looks like a very technical task due to the nature of the Steenrod generalized homology.
Since it is unclear if the above injectivity condition brings new classes of groups, this goal is not pursued in the paper.
\end{remark}

\section{Rationality conjecture}

Corollary~\ref{smallhomology} defines a subgroup of integral homology group $H_*(B\pi)$ of the macroscopically small homology classes $H_*^{sm}(B\pi)$
in spirit as Brunnbauer-Hanke defined corresponding subgroups for several other classes of manifolds~\cite{BH}. A major difference is that
Brunbauer and Hanke considered the rational homology. The following is closely related to the Gromov conjecture.

{\bf The Rationality Conjecture.}
{\em Macroscopically small homology classes of a group $\pi$ are rational.}

The precise meaning of this conjecture is that there are subgroups $H^{sm}_n(B\pi;\mathbb Q)\subset H_n(B\pi;\mathbb Q)$ such that for an orientable $n$ manifold $M$ the universal covering $\Wi M$ is macroscopically small if and only if $f_*([M])\in H^{sm}_n(B\pi;\mathbb Q)$. 
Clearly, $H^{sm}_n(B\pi;\mathbb Q)=H^{sm}_n(B\pi)\otimes\mathbb Q$.
Therefore, the Rationality Conjecture states that  if $f_*([M])$ is a torsion in $H_n(B\pi)$ then $\dim_{mc}\Wi M<n$. 

Note that
an affirmative answer to Problem~\ref{mac-large} implies the Rationality Conjecture.
\begin{thm}
The Rationality Conjecture implies the weak Gromov conjecture for spin manifolds whose fundamental group satisfies the  Novikov Conjecture.
\end{thm}
\begin{proof}
It follows from Rosenberg's theorem and the Chern character isomorphism for homology that $f_*([M])=0$ in $H_n(B\pi;\mathbb Q)$.
Since $0\in H_n^{sm}(B\pi)$, by the Rationality Conjecture $\Wi M$ might be macroscopically small.
\end{proof}

\subsection{The Stone-\v Cech compactification.}
Let $M$ be a closed $n$-manifold and let $f:M\to B\pi$ be a classifying map of its universal cover. Since $M$ is compact, 
the universal covering map $p:\Wi M\to M$ can be extended to the continuous map of the Stone-\v Cech compactification $\bar p:\Wi M\to M$.

\begin{thm}\label{cech} 
For a closed $n$-manifold $M$ the following conditions are equivalent:

(1) $\dim_{mc}\Wi M<n$ for the universal cover $\Wi M$ of $M$.

(2)  The map $f\circ\bar p:\beta(\Wi M)\to B\pi$ of the Stone-\v Cech compactification  can be deformed to the $(n-1)$-dimensional skeleton $B\pi^{(n-1)}$ of $B\pi$.
\end{thm}
\begin{proof}
$(1)\ \Rightarrow\ (2)$. Let $\dim_{mc}\Wi M<n$. By Proposition~\ref{deform} a lift $\Wi f:\Wi M\to E\pi$ of a classifying map $f:M\to B\pi$ to the universal covering can be deformed to the $(n-1)$-dimensional skeleton $E\pi^{(n-1)}$  by a bounded homotopy. Using the uniform contractibility of $E\pi$ and the nature of the metric on it (see the proof of Proposition 3.1 in~\cite{Dr1}) we can construct a homotopy $H:\Wi M\times I\to E\pi$
joining $\Wi f$ with a map $g:\Wi M\to E\pi^{(n-1)}$ such that for some fixed $\lambda>0$  the restrictions $H|_{x\times I}$ are $\lambda$-Lipschitz for all $x\in X$.  Since the metric on $B\pi$ is proper, it follows that
$p\circ H(\Wi M\times I)$ lies in a compact subcomplex $L$. By  the Ascoli-Arzela theorem the space $Map_{\lambda}(I, L)$ of all $\lambda$-Lipschitz maps $\phi:I\to L$ is compact. The map $p\circ H$ induces a continuous map $h:\Wi M\to Map_{\lambda}(I, L)$
to a compact space. Therefore it admits a continuous extension $\bar h:\beta(\Wi M)\to Map_{\lambda}(I, L)$ to the Stone-\v Cech compactification. The map $\bar h$ defines a homotopy $\bar H:\beta(\Wi M)\times I\to L$ that joins $f\circ\bar p$ with a map to $B\pi^{(n-1)}$.

$(2)\ \Rightarrow\ (1)$. We define a function $\Psi: B\pi^I\to\mathbb R$  on the space of all paths $B\pi^I$ as follows: for every path $\phi:I\to B\pi$ we set
$\Psi(\phi)$ to be the diameter of $\phi'(I)$ for a lift $\phi'$ of $\phi$. Since the metric on $E\pi$ is $\pi$-invariant, $\Psi(\phi)$
does not depend on the choice of $\phi'$. We leave to the reader to show that $\Psi$ is continuous.

A deformation $H:\beta(\Wi M)\to B\pi$ of $f\circ\bar p$ to a map $g:\beta(\Wi M)\to B\pi^{(n-1)}$ defines a continuous map $h:\beta(\Wi M)\to B\pi^I$. In view of compactness of $\beta(\Wi M)$ the function $\Psi\circ h$ is bounded. Then a lift of the homotopy $H|_{\Wi M\times I}:\Wi M\times I\to B\pi$ defines a bounded homotopy of $\Wi f$ to the $(n-1)$-skeleton $E\pi^{(n-1)}$. Proposition~\ref{deform} completes the proof.
\end{proof}

\begin{conjec}\label{sheaf}
For every simply connected  open $n$-manifold $N$ and any locally constant sheaf $\mathcal S$ on 
the Stone-\v Cech compactification $\beta(N)$  with  a free abelian group as the stalk, the cohomology group $H^n(\beta(N),\mathcal S)$ is torsion free.
\end{conjec}
Since $H^n(N)=0$, it follows from~\cite{CS} that $H^n(\beta(N),\mathcal S)=0$ in the case when $\mathcal S$ is a constant sheaf.
\begin{prop}
Conjecture~\ref{sheaf} implies the Rationality Conjecture.
\end{prop}
\begin{proof} Suppose that $f_*([M])=0$ in $H_n(B\pi;\mathbb Q)$. Thus, $f_*([M])$ is a torsion in $H_n(B\pi)$. Since $\dim \beta(\Wi M)=n$,
The primary obstruction $\tilde o$ to deform $f\circ\bar p:\beta(\Wi M)\to B\pi$ to $B\pi^{(n-1)}$ is the only obstruction and it lives in the $n$-dimensional
cohomology group with locally constant coefficient system $(f\circ\bar p)^*(\mathcal S)$ where $\mathcal S$ is the system on $B\pi$ that corresponds to the $\pi$-module $\pi_{n-1}(B\pi^{(n-1)})$. Note that the stalk of $\mathcal C$ is a free abelian group. Note that $\tilde o=\bar p^*(o_f)$ where $o_f$ is the primary obstruction to deform $f$ into $B\pi^{(n-1)}$. Since $f_*([M])$ is a torsion, it follows that $o_f$ has finite order. Therefore $\tilde o$ has finite order.
Then Conjecture~\ref{sheaf} implies that $\tilde o=0$. Then by Theorem~\ref{cech} $\dim_{mc}\Wi M<n$.
\end{proof}

\subsection{The Berstein-Schwarz cohomology class.}
We recall that the Berstein-Schwarz class~\cite{Sw},\cite{Be}  $b=b_{\pi}\in H^1(\pi,I(\pi))$ of a group $\pi$ is defined by the cochain
on the Cayley graph $\phi: G\to I(\pi)$ which takes an ordered  edge $[g,g']$ to $g'-g$. Here $I(\pi)$ is the augmentation ideal of
the group ring $\mathbb Z\pi$. We use notation $I(\pi)^k$ for the $k$-times tensor product $I(\pi)\otimes\dots\otimes I(\pi)$
over $\mathbb Z$. Then the cup product $b^k=b\smile\dots\smile b$ is defined as an element of $H^k(\pi,I(\pi)^k)$.

The Berstein-Schwarz class is universal in the following sense.
\begin{thm}[{[Sw], [DR]}]\label{universal}
For every $\pi$-module $L$ and every element $\alpha\in H^k(\pi,L)$ there is a $\pi$-homomorphism $\xi:I(\pi)^k\to L$ such that $\xi^*(b^k)=\alpha$ where $$\xi^*:H^k(\pi,I(\pi)^k)\to H^k(\pi,L)$$ is the coefficient homomorphism.
\end{thm}

By bringing in the Berstein-Schwarz class we can extend Theorem~\ref{small} to the following.

\begin{thm}\label{small2}
For a closed oriented $n$-manifold $M$ with the classifying map $f:M\to B\pi$ and its lift to the universal covers
$\Wi f:\Wi M\to E\pi$ the following are equivalent:

1. $\dim_{mc}\Wi M<n$;

2. $\Wi f_*([\Wi M])=0$ in $H^{lf}_n(E\pi;\mathbb Z)$ where $[\Wi M]\in H^{lf}_n(\Wi M;\mathbb Z)$ is the fundamental class of $\Wi M$;

3. $f_*([M])\in ker(ec_*^{\pi})$ where $[M]$ is the fundamental class of $M$;

4. $f^*(b^n)\in ker(ec^*_M)$ where $b$ is the Berstein-Schwarz class of $\pi$.
\end{thm}
\begin{proof}
1. $\Rightarrow$ 2. We may assume that $f:M\to B\pi$ is cellular  and Lipschitz 
for some metric CW complex structure on $B\pi$. If $\dim_{mc}\Wi M<n$, then by
Proposition~\ref{deform} there is a bounded cellular
homotopy of $\Wi f:\Wi X\to E\pi$ to a map $g:\Wi X\to E\pi^{(n-1)}$
with a compact projection to $B\pi$. By
Proposition~\ref{uniform=almost}, it is coarsely equivariant. Then
by Proposition~\ref{induced2} it follows that $\Wi
f_*([\Wi M]))=0$.

2. $\Rightarrow $ 3. $ec_*^{\pi}(f_*([M]))=\Wi
f_*(Pert_*^M([M]))=0$ and
hence, $f_*([M])\in ker(ec_*^{\pi})$.

3. $\Rightarrow $ 4.  If $f_*([M])\in ker(Pert_*^{\pi})$, then
$ec_*^{\pi}(f_*([M])\cap b^k)=0$. Since  the commutative
diagram
$$
\begin{CD}
H_0^{lf,ce}(\Wi M;I(\pi)^n) @>\bar f_*>> H_0^{lf,ce}(E\pi;I(\pi)^n)\\
@Aec_*^MAA @Aec_*^{\pi}AA\\
H_0(M;I(\pi)^n) @>f_*>> H_0(B\pi;I(\pi)^n)\\
\end{CD}
$$
has isomorphisms for horizontal arrows, $ec_*^M([M]\cap
(f^*(b^n))=0$. Thus, $ec_*^M([M])\cap ec^*_M(f^*(b^n))=0$.
By the Poincare Duality, $ec^*_M((b^n)=0$.

4. $\Rightarrow$ 1. We show that the obstruction $o_{\Wi f}$ to the inequality $\dim_{mc}\Wi M<n$ is  zero
and apply Theorem~\ref{obstr-dim}. By Theorem~\ref{universal} there is a $\pi$-homomorphism $\xi:I(\pi)^n\to L=\pi_{n-1}(B\pi^{(n-1)})$ such that $\xi^*(b^n)=\kappa_1$ and $\xi^*(f^*(b^n))=o_f$ where $\kappa_1$ is the primary obstruction to retract $B\pi$ onto $B\pi^{(n-1)}$ and $o_f$ is the primary obstruction do deform $f$ into $B\pi^{(n-1)}$. Then
$o_{\Wi f}=ec_M^*(o_f)=\xi^*ec_M^*(b^n)=0$.
\end{proof}

Every manifold $M$ with the fundamental group $\pi$ carries a local coefficients system $\mathcal I^n$ generated by the $\pi$-module $I(\pi)^n$.
\begin{prop}
Suppose that Conjecture~\ref{sheaf} holds true for the universal cover $\Wi M$ of a closed rationally inessential $n$-manifold $M$ with $\mathcal S=\bar p^*\mathcal I^n$. Then $\dim_{mc}\Wi M<n$. 
\end{prop}
\begin{proof}
In view of Theorem~\ref{cech} it suffices to show that
the primary obstruction $\tilde o$ to deform $f\circ\bar p:\beta(\Wi M)\to B\pi$ to $B\pi^{(n-1)}$ is zero.   Note that $\tilde o=\bar p^*f^*(o_1)$ where $o_1\in H^n(B\pi;\mathcal S)$ is the primary obstruction to retraction of $B\pi$ to $B\pi^{(n-1)}$. By the universality of the Berstein-Schwarz class there is a
morphism of local coefficient systems $\mathcal I^n\to\mathcal S$ over $B\pi$ such that the induced homomorphism of $n$th cohomology takes $b^n$ to $o_1$. Then $\bar p^*f^*(b^n)$ is taken to $\tilde o$ by the corresponding homomorphism.
Since $f_*([M])$ is a torsion, it follows that $f^*(b^n)$ has finite order. Therefore $\tilde o$ has finite order.
Then Conjecture~\ref{sheaf} for the  sheaf $\bar p^*\mathcal I^n$  implies that $\tilde o=0$.
\end{proof}

\end{document}